\documentclass[11pt]{amsart}

\usepackage{amssymb,comment}

\usepackage{subfigure, graphicx, tikz, color}
\usepackage{hyperref}     
\usepackage{graphicx} 
\usepackage[font=small,labelfont=bf]{caption}

\newcommand{\E}{\,{\mathbb{E}}}
\newcommand{\Prob}{{\mathbb{P}}}
\newcommand{\R}{\mathbb{R}}

\newcommand{\N}{{\mathbb N}}

\providecommand{\P}{}
\renewcommand{\P}[1]{{\mathbb P}\left\{#1\right\}}

\newcommand{\rnc}[1]{\textcolor{red}{#1}}
\newcommand{\flc}[1]{\textcolor{blue}{#1}}

\newtheorem{thm}{Theorem}
\newtheorem{lem}[thm]{Lemma}

\newtheorem{cor}[thm]{Corollary}
\newtheorem{dfn}[thm]{Definition}
\newtheorem{rem}[thm]{Remark}

\theoremstyle{remark}

\DeclareMathOperator{\dist}{dist} 
\DeclareMathOperator{\Var}{Var}
\DeclareMathOperator{\Cov}{Cov}

\usepackage[normalem]{ulem}

\definecolor{clou}{rgb}{0.8,0.25,0.5125}

\usepackage[numbers,sort&compress] {natbib}

\begin{document}


\title{On topological indices of random   caterpillar graphs}

\author[Lesny]{Florian Lesny}
\address{Institute of Mathematics, Goethe University Frankfurt, Frankfurt a.M., Germany}
\email{lesny@math.uni-frankfurt.de}

\author[Neininger]{Ralph Neininger}
\address{Institute of Mathematics, Goethe University Frankfurt, Frankfurt a.M., Germany}
\email{neininger@math.uni-frankfurt.de}

\author[Schmieder]{Markus Schmieder}
\address{Institute of Mathematics, Goethe University Frankfurt, Frankfurt a.M., Germany}
\email{s4238590@stud.uni-frankfurt.de}

\begin{abstract}
We consider the asymptotic behaviour of topological indices of random caterpillar graphs, such as the Zagreb, $\alpha$-Randić, Wiener, Gini and level indices, as the size of the graph grows. Our simple, general framework, which is based on a multivariate central limit theorem, the continuous mapping theorem and Slutsky's theorem, can also be used to cover other topological indices. We derive limit laws for these indices and provide exact and asymptotic results on their means and variances. The $\alpha$-Randić index exhibits a discontinuity in its limit distribution at $\alpha=\frac{1}{2}$.
\end{abstract}

\date{\today}

\maketitle



\section{Introduction and results}\label{introduction}
\noindent
Caterpillar graphs, first introduced by \cite{harary73}, are acyclic graphs used in mathematical chemistry to model molecular structures. They are characterised by a chain of nodes called the 'spine', which remains unchanged during the growth process, while new nodes called 'children' are attached to the spine nodes. We are particularly interested in the behaviour of random caterpillar graphs, in which the new nodes are connected to a parent spine node chosen uniformly at random. To analyse these random graph structures, we will discuss several topological indices, such as the Zagreb index \cite{gutman72}, the Randić index \cite{randic75}, the Wiener index \cite{wiener47} and the Gini index \cite{ceve12}. General references on topological indices of (random) graph models can be found in the monographs and collections \cite{tri92,gu86,de99}. For more detailed probabilistic analyses of such indices, including the derivation of limit distributions or tail bounds, see \cite{ne02,ja03,fu15,ma17,akne07}.

Expansions of the mean and variance for the aforementioned indices and some others have been developed for random caterpillars by \cite{zhang2021}. In this note, we derive limit theorems for these topological indices. 

This work builds upon the third author's BSc thesis and is organised as follows: First, we provide a brief introduction to the model and present our results. We also address an erroneous limit distribution reported in previous literature. Section \ref{mclt} contains the methods underlying our proofs. The proofs are presented in Section \ref{proofs}, while some technical calculations are collected in the Appendix. 

\subsection*{Random caterpillar}
A \emph{caterpillar} is a graph process that, initially (at time $0$) consists of $m$ nodes, which are arranged as a chain. We call them \emph{spine nodes}. Then, the graph grows in discrete time, at each time step by one node, which is attached to one of the $m$ spine nodes. We  call this new node a \emph{child} of the spine node to which it is connected. Note that at time $n$ (i.e., after $n$ steps) the graph has $n+m$ nodes and $(n+m-1)$ edges. A \emph{random caterpillar} is a caterpillar where, in each step, the spine node that gets connected to the new node is chosen uniformly at random from the $m$ spine nodes and independently from the previous choices. 

By 
$$I^{(n)}:=\left(I_1^{(n)},\dots,I_m^{(n)}\right)$$
we denote the numbers of children of the spine nodes at time $n$. We denote by $D^{(n)}_j$ the degree of the $j$th spine node at time $n$. Note, that the  $I^{(n)}$ determine these degrees: The first and last spine node have one adjacent spine node while all other spine nodes have two neighbors. This implies $D_1^{(n)}=I_1^{(n)}+1$, $D_m^{(n)}=I_m^{(n)}+1$ and $D_j^{(n)}=I_j^{(n)}+2$ for $j \in \{2,\dots,m-1\}$.\\
Note that $I^{(n)}$ is multinomially distibuted with $n$ trials and all event probabilities being $\frac{1}{m}$, which we denote by $I^{(n)}\overset{d}{=}M(n;\frac{1}{m},\dots,\frac{1}{m})$. 
\subsection*{Topological indices}
We are interested in topological indices of random caterpillar graphs. A topological index of a graph is a graph-invariant quantity that often provides information about the underlying graph's structure. All graphs are assumed to be undirected. We use the notation $d(v)$ for the degree of node $v$ and $\dist(u,v)$ as the graph distance (number of edges on the shortest path) between the nodes $u$ and $v$. The structure of the caterpillar implies $\dist(u,v) \le m+1$ for all nodes $u,v$. If the graph is rooted, we write $D(v)$ for the distance between the node $v$ and the root, also called the depth of node $v$.

\begin{dfn}\label{indices}
    Let $H=(V,E)$ be a graph with node set $V$ and edge set $E$.\\
    The Zagreb index is defined as
    \begin{align}\label{Zagreb}
        Z_H:=\sum_{v \in V} (d(v))^2.
    \end{align}
For $\alpha \in \R$, the  $\alpha$-Randić index of $H$ is given by
    \begin{align}\label{Randic}
        R_H^{[\alpha]}:=\sum_{\{u,v\} \in E} (d(u)d(v))^\alpha.
    \end{align}
If $H$ is connected, the Wiener index is defined as  
    \begin{align}\label{Wiener}
        W_H:=\frac{1}{2}\sum_{u,v\in V} \dist(u,v).
    \end{align}
If $H$ is connected and rooted, the level index is defined as 
    \begin{align}\label{level}
        L_{H}:=\frac{1}{2}\sum_{u,v \in V} \vert D(u)-D(v)\vert.
    \end{align}
For a random graph $H$ of a class of rooted, connected graphs $\mathcal{H}$, the Gini index is given by
\begin{align}\label{Gini}
G_H:=\frac{L_{H}}{\E[V]^2\E[D_\mathcal{H}^\star]},
\end{align}
where $D_\mathcal{H}^\star$ is the depth of a randomly chosen node in a random tree of class $\mathcal{H}$ (see \cite{ma17} for more information).
\end{dfn}
\noindent
Common choices for the parameter of the Randić index are  $\alpha=1$ or $\alpha=-1/2$; see \cite{randic75,zhang2021} for more details. For other versions of the Gini index, see \cite{ma17,zhde19}.
\subsection*{Notation}
By  $\stackrel{d}{\longrightarrow}$ and $\stackrel{\Prob}{\longrightarrow}$ we denote convergence in distribution and in probability, respectively. We use the Bachmann--Landau big $\mathrm{O}$ and little $o$ notation, where $o_\Prob(1)$ denotes a term which go to zero in probability. By $\mathcal{N}(0,\sigma^2)$, we denote the centered normal distribution with variance $\sigma^2$, and $a^t$ denotes the transposed vector (a column vector) of $a\in\R^m$.

\subsection*{Results}
We are now ready to present our results, which include limit theorems for the five topological indices defined above in a random caterpillar. Throughout the remainder of this note, $Z_n$, $W_n$, $R_n^{[\alpha]}$, $L_n$ and $G_n$ will always refer to the topological index of a random caterpillar at time $n$, with the length of the spine, $m$, suppressed from the notation. For $L_n$ and $G_n$, the root is one end of the spine.\\
The first topological index that we will consider is the Wiener index, $W_n$. Its expectation, for  $n,m \in \N$, was derived in \cite{zhang2021}:
 \begin{align*}
        \E[W_n] = \frac{n^2(m^2+6m-1)+n(m-1)(2m^2+7m-1)}{6m}+\frac{m(m^2-1)}{6}.
    \end{align*}
Note that for $m=1$ we have $W_n=n^2$. We calculate the variance of $W_n$ and prove the central limit theorem. 
\begin{thm}\label{clt:Wiener}
    For all $n \in \N$ and $m \ge 2$ we have
    \begin{align}\label{Var_W_n}
        \Var(W_n) &= \frac{n^3(m-1)(m-2)(m+1)(m+2)}{180m^2}\\
        &\quad\quad~+\frac{n^2(m-1)(m+1)(m^3+m^2-4m+11)}{90m^2}\nonumber\\
        &\quad\quad~+\frac{n(m-1)(m+1)(m^4-2m^3-7m^2+8m-18)}{180m^2}.\nonumber
   \end{align}
    
     For $m=2$ as $n\to\infty$, we have
    \begin{equation*}
        \frac{W_n-\E[W_n]}{n} \overset{d}{\longrightarrow} 3N_1N_2+N_1^2+N_2^2+\frac{1}{4}
    \end{equation*}
    with $(N_1,N_2)$ defined in \eqref{N_i}.
    
    For $m \ge 3$ it holds
    \begin{equation*}
            \frac{W_n-\mathbb{E}[W_n]}{n^{3/2}} \overset{d}{\longrightarrow}
            \mathcal{N}\left(0,\frac{(m-1)(m-2)(m+1)(m+2)}{180m^2}\right).
    \end{equation*}
\end{thm}
Analogous results for the level index $L_n$ and the Gini index $G_n$ are as follows. In \cite{ma17}, explicit formulae for the expectations were derived; see also \cite{zhde19} for other types of Gini indices and experimental results. For $m \in \N$ and $n \ge 0$, it was shown in \cite{ma17} that
\begin{align*}
    \E[L_n]=\frac{1}{6m}\left(\left(m^2-1\right)n^2+\left(2m^3-m^2+4m+1\right)n+m^4-m^2\right)
\end{align*}
and
\begin{align*}
    \E[G_n]=\frac{\left(m^2-1\right)n^2+\left(2m^3-m^2+4m+1\right)n+m^4-m^2}{6m(n+m)^2\left(n/(n+m)+(m-1)/2\right)}.
\end{align*}
To calculate the expectation of the Gini index, note that the number of nodes in the random caterpillar at time $n$ is $n+m$, and the average depth of a randomly selected node is given by $n/(n+m)+(m-1)/2$. For $m=1$, we have trivially $L_n=n$ and $G_n=1/(n+1)$. Furthermore, we have limit theorems for the level and the Gini index.
\begin{thm}\label{clt:Gini} 
For $m=2$, as $n \to \infty$, we have
\begin{align*}
\frac{L_n-\E[L_n]}{n} \overset{d}{\longrightarrow} N_1N_2+\frac{1}{4}
\end{align*}
and
\begin{align*}
    n(G_n-\E[G_n]) \overset{d}{\longrightarrow} N_1N_2+\frac{1}{4}
\end{align*}
with $(N_1,N_2)$ defined in \eqref{N_i}.\\
For $m \ge 3$, as $n \to \infty$, we have
\begin{align*}
    \frac{L_n-\E[L_n]}{n^{3/2}} \overset{d}{\longrightarrow} \mathcal{N}\left(0,\frac{(m-1)(m-2)(m+1)(m+2)}{180m^2}\right)
\end{align*}
and
\begin{align*}
    \sqrt{n}\left(G_n-\E[G_n]\right) \overset{d}{\longrightarrow} \mathcal{N}\left(0,\frac{(m-2)(m-1)(m+2)}{45(m+1)m^2}\right).
\end{align*}
\end{thm}
The expectation and the variance of the Zagreb index $Z_n$ for $m \ge 2$ were derived in \cite{zhang2021}: 
For $n \ge 0$ and for $m \ge 2$, it was shown that
    \begin{align*}
        \E[Z_n]=\frac{n^2}{m}+\frac{(6m-5)n}{m}+4m-6
    \end{align*}
    and
    \begin{align*}
        \Var(Z_n)=\frac{2n((m-1)n+3m-7)}{m^2}.
    \end{align*}
    Note that for  $m=1$, we have $Z_n=n(n+1)$.
\noindent
Furthermore, we are able to compute a 
limit theorem for  $Z_n$, which  is non-normal.
\begin{thm}\label{clt:Zagreb}
For  $m \ge 2$, as $n\to\infty$, we have
\begin{align}\label{clt:Z_n}
    \frac{Z_n-\E[Z_n]}{n} \overset{d}{\longrightarrow}\sum_{i=1}^m N_i^2  -\frac{m-1}{m},
\end{align}
with $(N_1,\ldots,N_m)$ defined in \eqref{N_i}.
\end{thm}
\begin{rem}
Since the right-hand side of the limit statement \eqref{clt:Z_n} is a sum of non-negative random variables shifted by a constant, we can conclude that the limiting distribution is not normal. This corrects a statement from \cite[Theorem 1, page 1780]{zhang2021}.
\end{rem}
\noindent
Analysing the $\alpha$-Randić index of a caterpillar is more difficult than analysing the other indices because the exponent  $\alpha$ makes it harder to calculate the moments explicitly. Therefore, we only derive explicit expressions for the variance when $\alpha=1$; the expectation in this case was  already given for all $m\ge 2$ in \cite{zhang2021} as 
\begin{align*}
        \E\left[R_n^{[1]}\right]= \frac{n^2(2m-1)+n(7m^2-10m+1)+4m^2(m-2)}{m^2}.
    \end{align*}
Note that for $m=1$ we have $ R_n^{[1]}=n^2$. We further obtain
\begin{lem}
    For  $m \ge 2$ we have
    \begin{align*}
\Var\left(R_n^{[1]}\right)&= \frac{2n^3(m-2)}{m^4}+\frac{n^2(3m^3+3m^2-30m+10)}{m^4}\\
        &\quad\quad~+\frac{n(17m^3-67m^2+28m-6)}{m^4}.
    \end{align*}
\end{lem}
\noindent
For general $\alpha\in\R$ we have asymptotic expansions of the mean as well as limit distributions. Note that the case $\alpha=0$ where $R_n^{[\alpha]}$ is deterministic is trivial but also covered by the following statement:
\begin{thm}\label{clt:Randic}
   For all $m\ge 2$ and all $\alpha\in \R$ we have
    \begin{align*}
        \E\left[R_n^{[\alpha]}\right]&=\frac{m-1}{m^{2\alpha}}n^{2\alpha}+\frac{1}{m^\alpha}n^{\alpha+1}\\
        &\quad +\left(\frac{2\alpha(m-1)}{m^\alpha}+(m-1)\frac{2\alpha+\alpha(\alpha-1)}{2m^\alpha}\right)n^\alpha+o(n^\delta)
    \end{align*}
    with $\delta=\max\{2\alpha,\alpha\}$.\\
    For $m=2$ and $(N_1,N_2)$ defined in \eqref{N_i} we have for all $\alpha \in \R$
    \begin{align*}
        \frac{R_n^{[\alpha]}-\E\big[R_n^{[\alpha]}\big]}{n^\alpha} \overset{d}{\longrightarrow} \frac{(2\alpha+\alpha(\alpha-1))}{2^\alpha}\left(N_1^2+N_2^2-\frac{1}{2}\right).
    \end{align*}
    Let $m \ge 3$.
    \begin{enumerate}
    \item[(i)] For $\alpha > \frac{1}{2}$ it holds that
    \begin{align*}
       \frac{R_n^{[\alpha]}-\E\big[R_n^{[\alpha]}\big]}{n^{2\alpha-1/2}} \overset{d}{\longrightarrow} %
       \mathcal{N}\left(0,\frac{2\alpha^2(m-2)}{m^{4\alpha}}\right).
    \end{align*}
    \item[(ii)] For $\alpha < \frac{1}{2}$ it holds that
    \begin{align*}
        \frac{R_n^{[\alpha]}-\E\big[R_n^{[\alpha]}\big]}{n^\alpha} \overset{d}{\longrightarrow} \frac{2\alpha+\alpha(\alpha-1)}{2m^{\alpha-1}}\left(\sum_{i=1}^m N_i^2-\frac{m-1}{m}\right).
    \end{align*}
    \item[(iii)] For $\alpha=\frac{1}{2}$ it holds that
    \begin{align*}
        \frac{R_n^{[1/2]}-\E\big[R_n^{[1/2]}\big]}{\sqrt{n}}\overset{d}{\longrightarrow} \frac{1}{2}(N_1+N_m)+\frac{3\sqrt{m}}{8}\left(\sum_{i=1}^m N_i^2-\frac{m-1}{m}\right).
    \end{align*}
    \end{enumerate}
\end{thm}

For $m\ge 3$ in Theorem \ref{clt:Randic}, note that the distribution of $\frac{1}{2}(N_1+N_m)$ appearing in the case $\alpha=\frac{1}{2}$ is given by the normal $\mathcal{N}(0,2\alpha^2(m-2)/(m^{4\alpha}))$ distribution appearing in the case $\alpha>\frac{1}{2}$. Therefore, in the case of $\alpha=\frac{1}{2}$, the limit distribution can be interpreted as having contributions from both cases $\alpha<\frac{1}{2}$ and $\alpha>\frac{1}{2}$. Notably, there is a discontinuity in the limit distribution at $\alpha=\frac{1}{2}$.
 
\section{The basic probabilistic tools}\label{mclt}
\noindent
Since the studied topological indices are continuous functions of the variables $I^{(n)}_1,\dots,I^{(n)}_m$ with the multivariate distribution $M(n;1/m,\dots,1/m)$, we will use the multivariate Central Limit Theorem in combination with the Continuous Mapping Theorem (in the form of our Lemma \ref{Continuous Mapping Theorem}) and Slutsky's Theorem (in the form of Lemma \ref{Slutsky})  to derive the distributional convergence results.
\begin{lem}\label{Central Limit Theorem}\cite[Thm. 29.5.]{billingsley95}
    Let $(X_i)_{i \in \N}$ be a sequence of i.i.d random variables in $\R^d$ with $\Vert X_1\Vert$ being square-integrable and having a regular covariance matrix $C$. Then, we have for the sum $S_n:=\sum_{i=1}^n X_i$
    \begin{align*}
        \frac{S_n-n\E{X_1}}{\sqrt{n}}\stackrel{d}{\longrightarrow} Y,
    \end{align*}
    where $L(Y)=\mathcal{N}(0,\Sigma)$ and $\mathcal{N}(\mu,C)$ denotes a $d$-dimensional normal distribution with mean $\mu$ and covariance matrix $C$.
\end{lem}
\begin{cor}\label{clt:I_n}
    Let $I^{(n)}:=\left(I_1^{(n)},\dots,I_m^{(n)}\right)\overset{d}{=}M(n;1/m,\dots,1/m)$. Then, with $\mu^{(n)}=(n/m,\ldots,n/m)^t$, we have
    \begin{align}\label{N_i}
        I^{(n),\ast}:=\frac{I^{(n)}-\mu^{(n)}}{\sqrt{n}}\overset{d}{\longrightarrow} N=(N_1,\ldots,N_m),
    \end{align}
    where $\mathcal{L}(N)=\mathcal{N}(0,\Sigma)$ with 
\begin{equation}\label{Covariance}
\Sigma = \left( \begin{array}{cccc}
    \frac{m-1}{m^2} &-\frac{1}{m^2}& \cdots & -\frac{1}{m^2} \\
    -\frac{1}{m^2} & \ddots & &\vdots\\
    \vdots &  & \ddots& -\frac{1}{m^2} \\
    -\frac{1}{m^2} & \dots & -\frac{1}{m^2}& \frac{m-1}{m^2}
    \end{array}
    \right).
\end{equation}
\end{cor}
\begin{proof}
Let $\{e_i:i \in \{1,\dots,m\}\}$ be the canonical basis of $\R^m$. Define a random variable $X$ through $\P{X=e_i}=1/m$ for all $i \in \{1,\dots,m\}$ and let $(X_i)_{i \in \N}$ be a sequence of i.i.d. copies of $X$. We have $I^{(n)}\overset{d}{=}\sum_{i=1}^n X_i$ and an easy calculation yields that the covariance matrix of $X$ is given by \eqref{Covariance}. Since $\Sigma$ given in \eqref{Covariance} is a regular matrix, the statement directly follows from Lemma \ref{Central Limit Theorem}. 
\end{proof}
Note that Corollary \ref{clt:I_n} has also been used for the analysis of random caterpillar graphs in \cite{zhde19}.

For further use and the reader's convenience we further recall the Continuous Mapping Theorem in the form used below.
\begin{lem}\label{Continuous Mapping Theorem}\cite[The Mapping Theorem, page 20]{bill99}
    Let $(X_n)_{n \in \N},X$ be $\R^m$ valued random variables and $f:\R^m \to \R$ a continuous function. If $X_n \overset{d}{\longrightarrow} X$, then we have $f(X_n) \overset{d}{\longrightarrow} f(X)$.
\end{lem}
To argue that certain terms are asymptotically negligible, we use the following form of Slutsky's theorem: 
\begin{lem}\label{Slutsky}
\cite[Thm.~3.1]{bill99}
Let $(X_n)_{n \in \N}$,$(Y_n)_{n \in \N}$,$X$ be real-valued random variables and $a \in \R$ a constant. If $X_n \overset{d}{\longrightarrow}X$ and $Y_n \overset{\Prob}{\longrightarrow} a$, then $X_n+Y_n \overset{d}{\longrightarrow}X+a$.
\end{lem}

\noindent

\section{Proofs}\label{proofs}
\noindent
We have tried to make the proofs as short and simple as possible. More detailed calculations can be found in the appendix.  To prove Theorems \ref{clt:Wiener}, \ref{clt:Gini} and \ref{clt:Zagreb}, we need  mixed moments of $I^{(n)}$.
\begin{lem}\label{mix_mom}
    For $I^{(n)}$ as in Corollary \ref{clt:I_n} we have
    \begin{align*}
    \E\left[{I_i^{(n)}}\left(I_i^{(n)}-1\right)\right]=\frac{n(n-1)}{m^2}
    \end{align*}
    and
    \begin{align*}
        \E\left[I_i^{(n)}I_j^{(n)}\right]=\frac{n(n-1)}{m^2} \text{ for } i \neq j.
    \end{align*}
\end{lem}
\begin{proof}
Since $I_i^{(n)}$ is binomial$(n,1/m)$ distributed, we have $\E\left[I_i^{(n)}\right]=n/m$ and $\Var\left(I_i^{(n)}\right)=n(m-1)/m^2$. This implies
\begin{align*}
\E\left[I_i^{(n)}\left(I_i^{(n)}-1\right)\right]=\frac{n(n-1)}{m^2}.
\end{align*}
Using $\Cov\left(I_i^{(n)},I_j^{(n)}\right)=-n/m^2$, we obtain
\begin{align*}
  \E\left[I_i^{(n)}I_j^{(n)}\right]=\Cov\left(I_i^{(n)}I_j^{(n)}\right)+\E\left[I_i^{(n)}\right]\E\left[I_j^{(n)}\right]=\frac{n(n-1)}{m^2}
\end{align*}
 for all $i \neq j$. 
\end{proof}
\noindent
To prove Theorem \ref{clt:Wiener}, we must first derive a more suitable form of the Wiener index in a caterpillar. This will allow us to exploit the asymptotic behaviour of $I_i^{(n),\ast}$ to obtain an asymptotic result for $W_n$.
\begin{lem}\label{Wiener:caterpillar}
    In the case of a caterpillar, we have
    \begin{align*}
    \begin{split}
        W_n=&\sum_{i<j} (j-i+2)I_i^{(n)}I_j^{(n)}+\sum_{i,j=1}^m (\vert j-i\vert+1)I_j^{(n)}\\
        &+\sum_{i=1}^m I_i^{(n)}(I_i^{(n)}-1)+\frac{(m-1)m(m+1)}{6}.
    \end{split}
    \end{align*}
\end{lem}
\begin{proof}
There are four different terms that contribute to the Wiener index in a caterpillar. Firstly, there are  the distances between the leaves of two different spine nodes. For two fixed spine nodes $i$ and $j$, the distance between their leaves is $|j-i|+2$. Since we are considering only unordered pairs, and since the spine node $i$ has $I_i^{(n)}$ children, we obtain the first summand.  The distance between spine node $i$ and the leaves of spine node $j$ is $\vert j-i\vert+1$, which gives us the second term. The distance between two leaves that are attached to the same spine node is $2$ and the number of ordered pairs in the cluster of spine node $i$ is $I_i^{(n)}\left(I_i^{(n)}-1\right)$. Finally, the sum of the distances between the spine nodes is equal to  $\sum_{i<j} j-i=((m-1)m(m+1))/6$.
\end{proof}
\begin{proof}[Proof of Theorem \ref{clt:Wiener}]
    The expectation of $W_n$ has been derived in \cite{zhang2021} and the variance can be found in Appendix \ref{Var:W_n}. Lemma \ref{Wiener:caterpillar} implies
    \begin{align*}
    \begin{split}
        W_n-\E[W_n]=&\sum_{i<j} (j-i+2)\left(I_i^{(n)}I_j^{(n)}-\E\left[I_i^{(n)}I_j^{(n)}\right]\right)\\
        &+\sum_{i,j=1}^m (\vert j-i\vert+1)\left(I_j^{(n)}-\E\left[I_j^{(n)}\right]\right)\\
        &+\sum_{i=1}^m I_i^{(n)}\left(I_i^{(n)}-1\right)-\E\left[I_i^{(n)}\left(I_i^{(n)}-1\right)\right].
        \end{split}
    \end{align*}
    To estimate the first sum in the latter display, we use $I_i^{(n)}=\frac{n}{m}+\sqrt{n}I_i^{(n),\ast}$ and  $\E[I_i^{(n)}I_j^{(n)}]=\frac{n(n-1)}{m^2}$ to obtain
    \begin{align*}
        \begin{split}
           I_i^{(n)}I_j^{(n)}-\E\left[I_i^{(n)}I_j^{(n)}\right]
           &=\left(\frac{n}{m}+\sqrt{n}I_i^{(n),\ast}\right)\left(\frac{n}{m}+\sqrt{n}I_j^{(n),\ast}\right)-\frac{n(n-1)}{m^2}\\
           &=\frac{n^{3/2}}{m}\left(I_i^{(n),\ast}+I_j^{(n),\ast}\right)+nI_i^{(n),\ast}I_j^{(n),\ast}+\frac{n}{m^2}
        \end{split}
    \end{align*}
    Using the same substitution for the second term yields 
    \begin{align*}
       I_j^{(n)}-\E\left[I_j^{(n)}\right]= \sqrt{n}I_j^{(n),\ast}.
    \end{align*}
    For the third term, we have
    \begin{align*}
    \begin{split}
        I_i^{(n)}\left(I_i^{(n)}-1\right)&-\E\left[I_i^{(n)}\left(I_i^{(n)}-1\right)\right]\\
        &=\left(\frac{n}{m}+\sqrt{n}I_i^{(n),\ast}\right)^2-\left(\frac{n}{m}+\sqrt{n}I_i^{(n),\ast}\right)-\frac{n(n-1)}{m^2}\\
        &=\frac{2n^{3/2}}{m}I_i^{(n),\ast}+n\left(I_i^{(n),\ast}\right)^2-\sqrt{n}I_i^{(n),\ast}-\frac{n(m-1)}{m^2}.
        \end{split}
    \end{align*}
Note that
    \begin{align}\label{sum:I_n}
        \sum_{i=1}^m I_i^{(n),\ast}=0
    \end{align}
    since the $I_i^{(n)}$ sum up to $n$. This implies
    \begin{align*}
    \begin{split}
        \sum_{i=1}^m I_i^{(n)}\left(I_i^{(n)}-1\right)&-\E\left[I_i^{(n)}\left(I_i^{(n)}-1\right)\right]\\
        &=\sum_{i=1}^m \left(n\left(I_i^{(n),\ast}\right)^2-\sqrt{n}I_i^{(n),\ast}\right)-\frac{n(m-1)}{m}.
        \end{split}
    \end{align*}
    Overall, we obtain
    \begin{align*}
        \begin{split}
        W_n-\E[W_n]=&\sum_{i<j} (j-i+2)\left(\frac{n^{3/2}}{m}\left(I_i^{(n),\ast}+I_j^{(n),\ast}\right)+nI_i^{(n),\ast}I_j^{(n),\ast}+\frac{n}{m^2}\right)\\
        &+\sum_{i,j=1}^m (\vert j-i\vert+1)\sqrt{n}I_i^{(n),\ast}\\
        &+\sum_{i=1}^m \left(n\left(I_i^{(n),\ast}\right)^2-\sqrt{n}I_i^{(n),\ast}\right)-\frac{n(m-1)}{m}.
        \end{split}
    \end{align*}
    Let $m \ge 3$. The first term in the sum is of order $n^{3/2}$ while all other terms are of smaller order. Therefore, with $W_n^\ast:=\frac{W_n-\E[W_n]}{n^{3/2}}$, since $I^{(n),\ast}$ converges in distribution towards $N$ given in \eqref{N_i}, we get
    \begin{align}\label{W_n_ast}
    \begin{split}
        W_n^\ast &=\frac{1}{m}\sum_{i<j} (j-i+2)\left(I_i^{(n),\ast}+I_j^{(n),\ast}\right)+o_\Prob(1)\\
        &\overset{d}{\longrightarrow} \frac{1}{m}\sum_{i<j} (j-i)(N_i+N_j)
        \end{split}
    \end{align}
    using the Continuous Mapping Theorem (Lemma \ref{Continuous Mapping Theorem}) and Slutsky's Theorem (Lemma \ref{Slutsky}) as well as $\sum_{i<y} I_i^{(n),\ast}+I_j^{(n),\ast}=(m-1)\sum_{i=1}^m I_i^{(n),\ast}=0$. The limit on the right-hand side of \eqref{W_n_ast} is of the form $\frac{1}{m}\sum_{k=1}^m a_k N_k$ with
    \begin{align}\label{a_k}
    \begin{split}
       a_k&=\sum_{i=1}^{k-1} {k-i}+\sum_{i=k+1}^m {i-k}\\
       &=\frac{k(k-1)}{2}+\frac{(m-k)(m-k+1)}{2}.
    \end{split}
    \end{align}
Since the $N_i$ are normal distributed with mean $0$, the arising limit distribution is also a centered normal distribution. From \eqref{Covariance} we obtain $\Var(N_i)=(m-1)/m^2$ and $\Cov(N_i,N_j)=-1/m^2$ for $i \neq j$. This implies
\begin{align*}
\begin{split}
    \Var\left(\frac{1}{2m}\sum_{k=1}^m c_k N_k\right)&=\frac{1}{4m^2}\left(\sum_{k=1}^m {c_k^2 \Var(N_k)}+\sum_{i \neq j} {c_i c_j \Cov(N_i,N_j)}\right)\\
    &=\frac{m-1}{4m^4}\sum_{k=1}^m {c_k^2}-\frac{1}{4m^4}\sum_{i \neq j} c_i c_j\\
    &=\frac{m-1}{4m^4}\sum_{k=1}^m {c_k^2} - \frac{1}{4m^4}\left(\sum_{k=1}^m c_k\right)^2+\frac{1}{4m^4}\sum_{k=1}^m c_k^2\\
    &=\frac{1}{4m^3} \sum_{k=1}^m c_k^2 - \frac{1}{4m^4}\left(\sum_{k=1}^m c_k\right)^2.
\end{split}
\end{align*}
Using $c_k=2a_k$ with $a_k$ given in \eqref{a_k}, we get $\sum_{k=1}^m c_k = \frac{2(m-1)m(m+1)}{3}$ and 
\begin{align*}
        \sum_{k=1}^m c_k^2&=\sum_{k=1}^m (k(k-1)+(m-k)(m-k+1))^2\\
        &=\frac{m(m+1)(m+2)(m^2+2m+2)}{15}.
\end{align*}
Putting these together, we obtain
\begin{align*}
    \begin{split}
        \Var&\left(\frac{1}{2m}\sum_{k=1}^m c_k N_k\right)\\
        &=\frac{m(m+1)(m+2)(m^2+2m+2)}{60m^3}- \frac{1}{4m^4}\left(\frac{2(m-1)m(m+1)}{3}\right)^2\\
        &=\frac{(m-2)(m-1)(m+1)(m+2)}{180m^2},
    \end{split}
\end{align*}
the assertion.\\
In the case $m=2$, the term of order $n^{3/2}$ vanishes because 
\begin{equation*}
    \sum_{i<j} (j-i+2)\left(I_i^{(n),\ast}+I_j^{(n),\ast}\right)=3\left(I_1^{(n),\ast}+I_2^{(n),\ast}\right)=0
\end{equation*}
using \eqref{sum:I_n}. Setting $W_n^\ast:=\frac{W_n-\E[W_n]}{n}$, the Continuous Mapping Theorem (Lemma \ref{Continuous Mapping Theorem}) together with Slutsky's Theorem (Lemma \ref{Slutsky}) gives us
\begin{align*}
W_n^{\ast}&=3\left(I_i^{(n),\ast}I_j^{(n),\ast}+\frac{1}{4}\right)+\left(I_i^{(n),\ast}\right)^2+\left(I_2^{(n),\ast}\right)^2-\frac{1}{2}+o_\Prob(1)\\
&\overset{d}{\longrightarrow}3N_1N_2+N_1^2+N_2^2+\frac{1}{4},
\end{align*}
the assertion for $m=2$.
\end{proof}
\noindent
The calculated variance of the normal distribution for $m\ge3$ rederives the leading term in the variance of $W_n$ (see \eqref{Var_W_n}).
\begin{proof}[Proof of Theorem \ref{clt:Gini}]
For a random caterpillar, the level index $L_n$ at time $n$ is given as 
\begin{align}\label{L_n}
    L_n=\sum_{i=0}^m \sum_{j=i+1}^m (j-i)S_i^{(n)} S_j^{(n)},
\end{align}
where $S_i^{(n)}$ is the number of nodes at depth $i$, so $S_0^{(n)}=1$, $S_i^{(n)}=I_i^{(n)}+1$ for $i=1,\dots,m-1$ and $S_m^{(n)}=I_m^{(n)}$. Substituting this into equation \eqref{L_n} and separating the term for $i = 0$ yields
\begin{align*}
    L_n=\sum_{j=1}^{m} jI_j^{(n)}+\sum_{j=1}^{m-1} j+\sum_{i=1}^{m}\sum_{j=i+1}^m (j-i)\left(I_i^{(n)}I_j^{(n)}+b_i I_j^{(n)}+b_j I_i^{(n)}+b_ib_j\right)
\end{align*}
with $b_i=1$ for $i=1,\dots,m-1$ and $b_m=0$. We get
\begin{align*}
   \lefteqn{ L_n-\E[L_n]}\\
   &=\sum_{j=1}^m j\left(I_j^{(n)}-\frac{n}{m}\right)
    +\sum_{i=1}^m \sum_{j=i+1}^m (j-i)\\
    &\qquad\times \bigg(I_i^{(n)}I_j^{(n)}-\E\left[I_i^{(n)}I_j^{(n)}\right]+b_i\left(I_j^{(n)}-\frac{n}{m}\right)+b_j\left(I_i^{(n)}-\frac{n}{m}\right)\bigg).
\end{align*}
Using again $I_i^{(n)}=\frac{n}{m}+\sqrt{n}I_i^{(n),\ast}$ we obtain similarly as in the proof of Theorem \ref{clt:Wiener} that 
\begin{align*}
    \lefteqn{L_n-\E[L_n]}\\
    &=\sum_{j=1}^m j\sqrt{n}I_j^{(n),\ast}\\
    &\quad\quad+\sum_{i=1}^m\sum_{j=i+1}^m (j-i) \Bigg(\frac{n^{3/2}}{m}\left(I_i^{(n),\ast}+I_j^{(n),\ast}\right)+nI_i^{(n),\ast}I_j^{(n),\ast}+\frac{n}{m^2}\Bigg)\\
&\quad\quad\quad\quad\quad\quad\quad\quad\quad\quad\quad+\sqrt{n}\left(b_iI_j^{(n),\ast}+b_jI_i^{(n),\ast}\right)\Bigg)
\end{align*}
Let $m \ge 3$ and set $L_n^\ast:=(L_n-\E[L_n]/n^{3/2}$. Since $I^{(n),\ast}$ converges in distribution towards a multivariate normal distributed random variable $N$, the Continuous Mapping Theorem (Lemma \ref{Continuous Mapping Theorem}) and Slutsky's Theorem (Lemma \ref{Slutsky}) imply
\begin{align*}
   L_n^\ast&=\frac{1}{m}\sum_{i=1}^m \sum_{j=i+1}^m (j-i)\left(I_i^{(n),\ast}+I_j^{(n),\ast}\right)+o_\Prob(1)\\
   &\overset{d}{\longrightarrow}\frac{1}{m}\sum_{i=1}^m \sum_{j=i+1}^m (j-i)(N_i+N_j).
\end{align*}
The right-hand side equals the limit in \eqref{W_n_ast} and we obtain the stated result for the level index.

If $m=2$, the term of order $n^{3/2}$ vanishes analogously to the Wiener index and we have, with $L_n^\ast:=(L_n-\E[L_n])/n$,
\begin{align*}
    L_n^\ast=I_1^{(n),\ast}I_2^{(n),\ast}+\frac{1}{4}+o_\Prob(1) \overset{d}{\longrightarrow} N_1N_2+\frac{1}{4}
\end{align*}
using the Continuous Mapping Theorem (Lemma \ref{Continuous Mapping Theorem}) and Slutsky's Theorem (Lemma \ref{Slutsky}).

The Gini index is given as 
\begin{align*}
    G_n&=\frac{L_n}{(n+m)^2(n/(n+m)+(m-1)/2)}\\
    &=\frac{2L_n}{(m+1)n^2+2m^2n+m^3-m^2}.
\end{align*}
For $m \ge 3$, setting $G_n^\ast:=\sqrt{n}(G_n-\E[G_n])$, we directly obtain
\begin{align*}
    G_n^\ast \overset{d}{\longrightarrow} \frac{2}{m(m+1)}\sum_{i=1}^m \sum_{j=i+1}^m (j-i)(N_i+N_j).
\end{align*}
So, the limit distribution is normal, with a mean of zero. The variance can be calculated directly from previous results, giving a value of 
$$\frac{4}{(m+1)^2}\frac{(m-2)(m-1)(m+1)(m+2)}{180m^2}=\frac{(m-2)(m-1)(m+2)}{45(m+1)m^2}.$$

For case $m=2$,  with $G_n^\ast:=n(G_n-\E[G_n])$, we have
\begin{align*}
    G_n^\ast \overset{d}{\longrightarrow} N_1N_2+\frac{1}{4},
\end{align*}
the assertion.
\end{proof}
\begin{proof}[Proof of Theorem \ref{clt:Zagreb}]
Using \eqref{Zagreb}, we obtain for the Zagreb index $Z_n$ of a random caterpillar at time $n$ that
    \begin{align*}
        Z_n=\sum_{i=1}^m \left(D_i^{(n)}\right)^2+n,
    \end{align*}
since the degree of each leaf is $1$ and the total number of leaves is $n$. 
The relation between $D_i^{(n)}$ and $I_i^{(n)}$ implies
\begin{align*}
\begin{split}
Z_n&=\left(I_1^{(n)}+1\right)^2+\left(I_m^{(n)}+1\right)^2+\sum_{i=2}^{m-1} \left(I_i^{(n)}+2\right)^2+n\\
&=\sum_{i=1}^m\left(I_i^{(n)}\right)^2+2\sum_{i=2}^{m-1} I_i^{(n)}+3n+4(m-2)+2.
\end{split}
\end{align*}
With $Z_n^\ast:=\frac{Z_n-\E[Z_n]}{n}$ we have
\begin{align}\label{Z_n_ast}
    Z_n^\ast=\frac{1}{n}\left(\sum_{i=1}^m {\left(I_i^{(n)}\right)^2-\E\left[\left(I_i^{(n)}\right)^2\right]}+2\sum_{i=2}^{m-1} {I_i^{(n)}-\E\left[I_i^{(n)}\right]} \right).
\end{align}
Once again substituting $I_i^{(n)}=n/m+\sqrt{n} I_i^{(n),\ast}$ gives us
\begin{align*}
\begin{split}
\left(I_i^{(n)}\right)^2-\E\left[\left(I_i^{(n)}\right)^2\right]&=\frac{n^2}{m^2}+2\frac{n^{3/2}}{m}I_i^{(n),\ast}+n\left(I_i^{(n),\ast}\right)^2-\frac{n^2}{m^2}-\frac{n(m-1)}{m^2}\\
&=2\frac{n^{3/2}}{m}I_i^{(n),\ast}+n\left(I_i^{(n),\ast}\right)^2-\frac{n(m-1)}{m^2}.
\end{split}
\end{align*}
Since $\sum_{i=1}^m I_i^{(n),\ast}=0$ and $I^{(n),\ast}$ converges in distribution towards a multivariate normal distributed random variable, we have with the Continuous Mapping Theorem (Lemma \ref{Continuous Mapping Theorem})
\begin{align*}
    \frac{1}{n}\sum_{i=1}^m {\left(I_i^{(n)}\right)^2-\E\left[\left(I_i^{(n)}\right)^2\right]}=\sum_{i=1}^m{\left(I_i^{(n),\ast}\right)^2}-\frac{m-1}{m} \overset{d}{\longrightarrow}\sum_{i=1}^m {N_i^2} - \frac{m-1}{m}.
\end{align*}
For the second term we have
\begin{align*}
    \frac{1}{n}\sum_{i=2}^{m-1} I_i^{(n)}-\E\left[I_i^{(n)}\right]=\frac{1}{n}\sum_{i=2}^{m-1}\sqrt{n}I_i^{(n),\ast} \overset{\Prob}{\longrightarrow} 0.
\end{align*}
With Slutsky's Theorem (Lemma \ref{Slutsky}) we finally obtain
\begin{align*}
    Z_n^\ast \overset{d}{\longrightarrow} \sum_{i=1}^m N_i^2 -\frac{m-1}{m},
\end{align*}
the assertion.
\end{proof}
\begin{proof}[Proof of Theorem \ref{clt:Randic}]
        Note that for a random caterpillar, the $\alpha$-Randić index is given as
    \begin{align*}
    \begin{split}
    R_n^{[\alpha]}&=\sum_{i=2}^m \left(D_{i-1}^{(n)}D_i^{(n)}\right)^\alpha+\sum_{i=1}^m I_i^{(n)}\left(D_i^{(n)}\right)^\alpha\\
    &=\sum_{i=2}^m \left(\left(I_{i-1}^{(n)}+c_{i-1}\right)\left(I_i^{(n)}+c_i\right)\right)^\alpha+\sum_{i=1}^m I_i^{(n)}\left(I_i^{(n)}+c_i\right)^\alpha\\
    &=:g(I_1^{(n)},\ldots,I_m^{(n)})
    \end{split}
\end{align*}
with $c_1=c_m=1$, $c_i=2$ for $i=2,\dots,m-1$ and the function $g:\R^m\to\R$ given by
\begin{equation*}
g\left(x_1,\dots,x_m\right)=\sum_{i=2}^m \left(\left(x_{i-1}+c_{i-1}\right)\left(x_i+c_i\right)\right)^\alpha+\sum_{i=1}^m x_i\left(x_i+c_i\right)^\alpha.
\end{equation*}
Using a Taylor expansion of $g$ around $\mu^{(n)}=(n/m,\dots,n/m)^t$, we obtain
\begin{align*}
 g\left(I^{(n)}\right)
    &=g\left(\mu^{(n)}\right)+\sum_{i=1}^m \frac{\partial g}{\partial x_i}(\mu^{(n)})\left(I_i^{(n)}-\mu_i^{(n)}\right)\\
    &\quad~+\frac{1}{2}\sum_{i=1}^m \sum_{j=1}^m \frac{\partial^2 g}{\partial x_i \partial x_j}(\mu^{(n)})\left(I_i^{(n)}-\mu_i^{(n)}\right)\left(I_j^{(n)}-\mu_j^{(n)}\right)\\
    &\quad~+\frac{1}{6}\sum_{i=1}^m \sum_{j=1}^m \sum_{k=1}^m \frac{\partial^3 g}{\partial x_i \partial x_j \partial x_k}\left(\tilde{I}^{(n)}\right)\left(I_i^{(n)}-\mu^{(n)}_i\right)\\&\quad\quad\quad\quad\quad\quad\quad\quad\quad\times\left(I_j^{(n)}-\mu^{(n)}_j\right)\left(I_k^{(n)}-\mu^{(n)}_k\right),
\end{align*}
where $\tilde{I}^{(n)}$ is component-wise between $\mu^{(n)}$ and $I^{(n)}$. 

To analyse the asymptotic behavior of $R_n^{[\alpha]}$, we study the asymptotics of the terms of the latter Taylor expansion, calling them the constant, linear, quadratic and third Taylor terms, respectively.  Defining $y_i:=x_i+c_i$ for $i\in\{1,\ldots,m\}$, we have
\begin{align*}
\begin{split}
    \frac{\partial g}{\partial x_1}(x)&=\alpha y_2(y_1y_2)^{\alpha-1}+y_1^\alpha+\alpha x_1 y_1^{\alpha-1},\\
    \frac{\partial g}{\partial x_m}(x)&=\alpha y_{m-1}\left(y_{m-1}y_m\right)^{\alpha-1}+y_m^\alpha+\alpha x_m y_m^{\alpha-1},\\
    \frac{\partial g}{\partial x_i}(x)&=\alpha y_{i-1} (y_{i-1}y_i)^{\alpha-1}+\alpha y_{i+1}(y_iy_{i+1})^{\alpha-1}+y_i^\alpha+\alpha x_i y_i^{\alpha-1}
    \end{split}
\end{align*}
for $i\in\{2,\ldots,m-1\}$. This implies
\begin{align*}
    \begin{split}
    \frac{\partial g}{\partial x_1}\big(\mu^{(n)}\big)&=\frac{\partial g}{\partial x_m}\big(\mu^{(n)}\big)\\
    &=\frac{\alpha}{m^{2\alpha-1}}n^{2\alpha-1}+\frac{(1+\alpha)}{m^{\alpha}}n^\alpha+\mathrm{O}\left(n^{2\alpha-2}\right)+\mathrm{O}\left(n^{\alpha-1}\right),\\
    \frac{\partial g}{\partial x_i}
    \big(\mu^{(n)}\big)&=\frac{2\alpha}{m^{2\alpha-1}}n^{2\alpha-1}+\frac{(1+\alpha)}{m^{\alpha}}n^\alpha+\mathrm{O}\left(n^{2\alpha-2}\right)+\mathrm{O}\left(n^{\alpha-1}\right)
    \end{split}
\end{align*}
and therefore, we obtain the linear Taylor term using $I_i^{(n)}-\mu_i^{(n)}=\sqrt{n}I_i^{(n),\ast}$ as
\begin{align*}
\lefteqn{\sum_{i=1}^m \frac{\partial g}{\partial x_i}\big(\mu^{(n)}\big)\left(I_i^{(n)}-\mu_i^{(n)}\right)}\\
&=\frac{\alpha}{m^{2\alpha-1}}n^{2\alpha-1/2}\left(\sum_{i=1}^m I_i^{(n),\ast}+\sum_{i=2}^{m-1} I_i^{(n),\ast}\right)\\
&\quad\quad~+\frac{(1+\alpha)}{m^{\alpha}}n^{\alpha+1/2}\sum_{i=1}^m I_i^{(n),\ast}+\mathrm{O}\left(n^{2\alpha-3/2}\right)+\mathrm{O}\left(n^{\alpha-1/2}\right).
\end{align*}
Since $\sum_{i=1}^m I_i^{(n),\ast}=0$, we obtain
\begin{align*}
\lefteqn{
    \sum_{i=1}^m \frac{\partial g}{\partial x_i}\big(\mu^{(n)}\big)\left(I_i^{(n)}-\mu_i^{(n)}\right)}\\
    &=\frac{\alpha}{m^{2\alpha-1}}n^{2\alpha-1/2}\sum_{i=2}^{m-1} I_i^{(n),\ast}+\mathrm{O}\left(n^{2\alpha-3/2}\right)+\mathrm{O}\left(n^{\alpha-1/2}\right).
\end{align*}
Note, that for $\alpha > 0$ we have $2\alpha-1/2>\alpha-1/2$; thus, the terms of order $n^{\alpha-1/2}$ are not dominant.\\
In a next step, we will study the quadratic Taylor term. Therefore, we have to calculate the second partial derivatives, which are given as
\begin{align*}
    \frac{\partial^2g}{\partial x_1^2}(x)&=\alpha(\alpha-1)y_2^2(y_1y_2)^{\alpha-2}+2\alpha y_1^{\alpha-1}+\alpha(\alpha-1) x_1 y_1^{\alpha-2},\\
    \frac{\partial^2 g}{\partial x_m^2}(x)&=\alpha(\alpha-1)y_{m-1}^2(y_{m-1}y_m)^{\alpha-2}+2\alpha y_m^{\alpha-1}+\alpha(\alpha-1)x_my_m^{\alpha-2},\\
 \frac{\partial^2 g}{\partial x_i^2}(x)&=\alpha(\alpha-1)\left(y_{i-1}^2(y_{i-1}y_i)^{\alpha-2}+y_{i+1}^2(y_i y_{i+1})^{\alpha-2}\right)\\
    &\quad\quad~+2\alpha y_i^{\alpha-1}+\alpha(\alpha-1)x_i y_i^{\alpha-2},\\
     \frac{\partial^2 g}{\partial x_{i} \partial x_{i-1}}(x)&=\frac{\partial^2 g}{\partial x_i \partial x_{i+1}}(x)=\alpha^2(y_{i-1}y_i)^{\alpha-1},
\end{align*}
for $i \in \{2,\dots,m-1\}$. All other second partial derivatives are zero. This implies, with $\gamma:=\max\{2\alpha-3,\alpha-2\}$, that
\begin{align*}
    \frac{\partial^2 g}{\partial x_1^2}\big(\mu^{(n)}\big)= \frac{\partial^2 g}{\partial x_m^2}\big(\mu^{(n)}\big)&=\frac{\alpha(\alpha-1)}{m^{2\alpha-2}}n^{2\alpha-2}+\frac{2\alpha+\alpha(\alpha-1)}{m^{\alpha-1}} n^{\alpha-1}+\mathrm{O}\left(n^\gamma\right),\\
     \frac{\partial^2 g}{\partial x_i^2}\big(\mu^{(n)}\big)&=\frac{2\alpha(\alpha-1)}{m^{2\alpha-2}}n^{2\alpha-2}+\frac{2\alpha+\alpha(\alpha-1)}{m^{\alpha-1}} n^{\alpha-1}+\mathrm{O}\left(n^\gamma\right),
\end{align*}
and
\begin{align*}
    \frac{\partial^2 g}{\partial x_{i} \partial x_{i-1}}\big(\mu^{(n)}\big)&= \frac{\partial^2 g}{\partial x_i \partial x_{i+1}}\big(\mu^{(n)}\big)=\frac{\alpha^2}{m^{2\alpha-2}}n^{2\alpha-2}+\mathrm{O}\left(n^{2\alpha-3}\right),
\end{align*}
for all $i\in\{2,\ldots,m-1\}$.
The quadratic Taylor term can now be computed as:
\begin{align*}
\lefteqn{
        \frac{1}{2}\sum_{i=1}^m \sum_{j=1}^m \frac{\partial^2 g}{\partial x_i \partial x_j}\big(\mu^{(n)}\big)\left(I_i^{(n)}-\mu_i^{(n)}\right)\left(I_j^{(n)}-\mu_j^{(n)}\right)}\\
        &=\frac{\alpha(\alpha-1)}{2m^{2\alpha-2}}n^{2\alpha-1}\left(\sum_{i=1}^m \left(I_i^{(n),\ast}\right)^2+\sum_{i=2}^{m-1} \left(I_i^{(n),\ast}\right)^2\right)\\
        &\quad\quad+\frac{2\alpha+\alpha(\alpha-1)}{2m^{\alpha-1}} n^{\alpha}\sum_{i=1}^m \left(I_i^{(n),\ast}\right)^2\\
        &\quad\quad+\frac{2\alpha^2}{m^{2\alpha-2}}n^{2\alpha-1}\sum_{i=2}^{m} I_{i-1}^{(n),\ast}I_i^{(n),\ast}+\mathrm{O}\left(n^{\gamma+1}\right).
\end{align*}
Note that $2\alpha-1<2\alpha-1/2$, so terms of order $n^{2\alpha-1}$ are dominated by the linear Taylor term for all $\alpha \in \R$. Chernoff's inequality, see \cite{ch52} or \cite[Theorem 1.1]{McD98}, implies that
\begin{align}\label{Chernoff}
    \begin{split}
 \Prob\left(\max_{i=1,\dots,m} \left\vert I_i^{(n)}-\frac{n}{m}\right\vert \ge n^{3/5} \right) &\le \sum_{i=1}^m \Prob\left(\left\vert I_i^{(n)}-\frac{n}{m} \right\vert \ge n^{3/5} \right)\\
        &=m\Prob\left(\left\vert I_1^{(n)}-\frac{n}{m}\right\vert \ge n^{3/5} \right)\\
        &\le 2m \exp\left(-2n^{1/5}\right).
    \end{split}
\end{align}
Hence, with high probability (w.h.p.), the components of the vector $I^{(n)}$ are of linear order. This implies that the third partial derivatives, $\frac{\partial^3g}{\partial x_i \partial x_j \partial x_k}(\tilde{I}^{(n)})$, are of order $n^\gamma$ or smaller w.h.p. Therefore, the third Taylor term is also at most of order $n^{\gamma+9/5}$ w.h.p.~More specifically, there exists a constant $C>0$ such that the absolute value of the third Taylor term is smaller than $Cn^{\gamma+9/5}$ with a probability of at least $1-2m\exp(-2n^{1/5})$. In all cases, it is contained within $o_\Prob(1)$.

To calculate the expectation of $R_n^{[\alpha]}$, we can now use the Taylor expansion of $g$ to obtain
\begin{align*}
\E\left[R_n^{[\alpha]}\right]=&g\left(\mu^{(n)}\right)\\
&+\frac{1}{2}\sum_{i=1}^m \sum_{j=1}^m \frac{\partial^2g}{\partial x_i \partial x_j}\left(\mu^{(n)}\right) \E\left[\left(I_i^{(n)}-\mu_i^{(n)}\right)\left(I_j^{(n)}-\mu_j^{(n)}\right)\right]\\
    &+\frac{1}{6}\sum_{i=1}^m \sum_{j=1}^m \sum_{k=1}^m \E\bigg[\frac{\partial^3g}{\partial x_i \partial x_j \partial x_k} \left(\tilde{I}^{(n)}\right)\left(I_i^{(n)}-\mu_i^{(n)}\right)\\
    &\quad\quad\quad\quad\quad\quad\quad\quad\quad\quad\quad \quad\quad\times\left(I_j^{(n)}-\mu_j^{(n)}\right)\left(I_k^{(n)}-\mu_k^{(n)}\right)\bigg],
\end{align*}
since $\E[I_i^{(n)}]=\mu_i^{(n)}$ for all $i=1,\dots,n$. Using Newton's generalized binomial Theorem we get
\begin{align*}
    g\left(\mu^{(n)}\right)&=\sum_{i=2}^{m} \left(\frac{n}{m}\right)^{2\alpha}+\sum_{i=1}^m \left(\left(\frac{n}{m}\right)^{\alpha+1}+\frac{\alpha c_i}{m^\alpha}n^\alpha+\mathrm{O}\left(n^{\alpha-1}\right)\right)+\mathrm{O}\left(n^{2\alpha-1}\right)\\
    &=\frac{m-1}{m^\alpha}n^{2\alpha}+\frac{1}{m^\alpha}n^{\alpha+1}+\frac{2\alpha(m-1)}{m^\alpha}n^\alpha+\mathrm{O}\left(n^{\alpha-1}\right)+\mathrm{O}\left(n^{2\alpha-1}\right).
\end{align*}
Since $I_i^{(n)}-\mu_i^{(n)}=\sqrt{n}I_i^{(n),\ast}$, we have
\begin{align*}
    \frac{1}{2}\sum_{i=1}^m \sum_{j=1}^m &\frac{\partial^2g}{\partial x_i \partial x_j}\left(\mu^{(n)}\right) \E\left[\left(I_i^{(n)}-\mu_i^{(n)}\right)\left(I_j^{(n)}-\mu_i^{(n)}\right)\right]\\
    &=\frac{2\alpha+\alpha(\alpha-1)}{2m^{\alpha-1}} n^{\alpha}\sum_{i=1}^m \E\left[\left(I_i^{(n),\ast}\right)^2\right]+\mathrm{O}\left(n^{2\alpha-1}\right)+\mathrm{O}\left(n^{\alpha-1}\right)\\
    &=\frac{2\alpha+\alpha(\alpha-1)}{2m^{\alpha-1}} n^{\alpha}\sum_{i=1}^m \frac{m-1}{m^2}+\mathrm{O}\left(n^{2\alpha-1}\right)+\mathrm{O}\left(n^{\alpha-1}\right)\\
    &=(m-1)\frac{2\alpha+\alpha(\alpha-1)}{2m^{\alpha}} n^{\alpha}+\mathrm{O}\left(n^{2\alpha-1}\right)+\mathrm{O}\left(n^{\alpha-1}\right).
\end{align*}
To bound the expectation of the third Taylor term, we observe that $(I_i^{(n)}-\mu_i^{(n)})(I_j^{(n)}-\mu_j^{(n)})(I_k^{(n)}-\mu_k^{(n)})$ is at most of order $n^3$ and the third partial derivatives of $g$ evaluated at $\tilde{I}^{(n)}$ are at most constant if $\gamma<0$ and of order $n^\gamma$ if $\gamma>0$. Using \eqref{Chernoff}, we obtain 
\begin{eqnarray*}
\lefteqn{
    \Bigg\vert\E\Bigg[\frac{\partial^3g}{\partial x_i \partial x_j \partial x_k} \left(\tilde{I}^{(n)}\right)\left(I_i^{(n)}-\mu_i^{(n)}\right)\left(I_j^{(n)}-\mu_j^{(n)}\right)\left(I_k^{(n)}-\mu_k^{(n)}\right)\Bigg]\Bigg\vert}\\
    &&~ \le C_1 n^{\gamma+1.8}+2C_2mn^{\eta+3}\exp\left(-2n^{0.2}\right)=\mathrm{O}\left(n^{\gamma+1.8}\right)
\end{eqnarray*}
with constants $C_1,C_2>0$ and $\eta:=\max\{\gamma,0\}$. Hence, this term converges to zero in all cases.\\ 

Ad (i): If $\alpha>\frac{1}{2}$, we have $2\alpha-\frac{1}{2}>\alpha$; hence, the linear Taylor term is dominant. Let 
$$R_n^{[\alpha],\ast}:=\frac{R_n^{[\alpha]}-\E[R_n^{[\alpha]}]}{n^{2\alpha-1/2}}.$$
Now the terms of order $n^{2\alpha}$ and $n^{\alpha+1}$ in $g\left(\mu^{(n)}\right)$ and $\E[R_n^{[\alpha]}]$ cancel, and we obtain
\begin{align*}
    R_n^{[\alpha],\ast}=\frac{\alpha}{m^{2\alpha-1}}\sum_{i=2}^{m-1}I_i^{(n),\ast}+o_\Prob(1)\overset{d}{\longrightarrow} \frac{\alpha}{m^{2\alpha-1}}\sum_{i=2}^{m-1} N_i
\end{align*}
using the Continuous Mapping Theorem (Lemma \ref{Continuous Mapping Theorem}) and Slutsky's Theorem (Lemma \ref{Slutsky}). Furthermore, we have $\sum_{i=2}^{m-1} N_i = N_1+N_m$ almost surely, so the latter limit is normal with mean $0$, and we obtain, using Lemma \ref{mix_mom}, the variance
\begin{align*}
    \begin{split}
        \Var\left(\frac{\alpha}{m^{2\alpha-1}}\left(N_1+N_m\right)\right)&=\frac{\alpha^2}{m^{4\alpha-2}}\left(2\Var(N_1)+2\Cov(N_1,N_m)\right)\\
        &=\frac{2\alpha^2}{m^{4\alpha-2}}\left(\frac{m-1}{m^2}-\frac{1}{m^2}\right)\\
        &=\frac{2\alpha^2}{m^{4\alpha}}\left(m-2\right).
    \end{split}
\end{align*}

Ad (ii): If $\alpha<\frac{1}{2}$, the term of order $n^\alpha$ in the quadratic Taylor term is dominant, and with 
$$R_n^{[\alpha],\ast}:=\frac{R_n^{[\alpha]}-\E[R_n^{[\alpha]}]}{n^\alpha}$$
we obtain
\begin{align*}
    R_n^{[\alpha],\ast}&=\frac{2\alpha+\alpha(\alpha-1)}{2m^{\alpha-1}}\sum_{i=1}^m \left(\left(I_i^{(n),\ast}\right)^2-\E\left[\left(I_i^{(n),\ast}\right)^2\right]\right)+o_\Prob(1)\\
    &\overset{d}{\longrightarrow} \frac{2\alpha+\alpha(\alpha-1)}{2m^{\alpha-1}}\left(\sum_{i=1}^m N_i^2-\frac{m-1}{m}\right)
\end{align*}
using once again the Continuous Mapping Theorem (Lemma \ref{Continuous Mapping Theorem}) and Slutsky's Theorem (Lemma \ref{Slutsky}) as well as $\E[(I_i^{(n),\ast})^2]=(m-1)/m^2$ for all $i=1,\dots,m$ as given in \eqref{Covariance}.\\

Ad (iii): If $\alpha=\frac{1}{2}$, then $2\alpha-\frac{1}{2}=\alpha$ and the linear and quadratic Taylor terms are both of order $\sqrt{n}$. Therefore, the asymptotic distribution of $R_n^{[1/2],\ast}$ is the sum of the two cases: when $\frac{1}{2} < \alpha < 1$ and when $0 < \alpha < \frac{1}{2}$.

If $m = 2$, the linear Taylor term also vanishes when $\frac{1}{2} \le \alpha$, since in this case, we have $c_1 = c_2 = 1$ and $I_1^{(n),\ast} + I_2^{(n),\ast} = 0$. Consequently, the quadratic Taylor term is dominant for all values of $\alpha$, which yields the aforementioned result for $m=2$.
\end{proof}

\bibliographystyle{amsplain}

\bibliography{ref}
\pagebreak

\appendix
\section{Variance of $W_n$}\label{Var:W_n}
\noindent
In order to calculate the variance of  $W_n$, we must first estimate the variances of the individual terms and their covariances. To do this, we need the mixed moments of a multinomial distribution, which we take from \cite{ouimet2020}. To keep the text readable, the calculations will be split into several lemmata. Furthermore, we will use the notation $I_i$ instead of $I_i^{(n)}$.
\begin{lem}
    We have 
    \begin{align*}
    \begin{split}
    \Var\left(\sum_{i<j} (j-i+2)I_iI_j\right)=&\frac{n^3(m-1)(m-2)(m+1)(m+2)}{180m^2}
        \\
        &+\frac{n^2(m-1)(m^3+m^2+71m+251)}{90m^2}
        \\
        &-\frac{n(m-1)(m^3+m^2+46m+166)}{60m^2}.
        \end{split}
        \end{align*}
\end{lem}
\begin{proof}
First of all, the sum can be rewritten as
\begin{align*}
    \sum_{i<j} (j-i+2)I_iI_j=\sum_{i=1}^{m-1}\sum_{j=1}^{m-i}(j+2)I_iI_{i+j}.
\end{align*}
Then we have
\begin{align*}
    \begin{split}
        \E\left[\sum_{i=1}^{m-1}\sum_{j=i+1}^{m}(j-i+2)I_iI_j\right]&=\sum_{i=1}^{m-1}\sum_{j=1}^{m-i}(j+2)\E\left[I_iI_{i+j}\right]\\
    &= \E\left[I_1I_2\right]\sum_{i=1}^{m-1}\frac{(m-i)(m-i+5)}{2}\\
    &= \frac{n(n-1)}{m^2}\sum_{i=1}^{m-1}\frac{i(i+5)}{2}
    \\
    &= \frac{n(n-1)(m-1)(m+7)}{6m}
    \end{split}
\end{align*}
and
\begin{align*}
    \begin{split}
        \E&\left[\left(\sum_{i=1}^{m-1}\sum_{j=i+1}^{m}(j-i+2)I_iI_j\right)\left(\sum_{k=1}^{m-1}\sum_{l=k+1}^{m}(l-k+2)I_kI_l\right)\right] 
        \\
        &= \sum_{i=1}^{m-1}\sum_{j=i+1}^m\sum_{k=1}^{m-1}\sum_{l=k+1}^m(j-i+2)(l-k+2)\E\left[I_iI_jI_kI_l\right]
        \\
        &=\sum_{i=1}^{m-1}\sum_{j=i+1}^m\sum_{k=1,k\neq i,j}^{m-1}\sum_{l=k+1,l\neq i,j}^m(j-i+2)(l-k+2)\E\left[I_iI_jI_kI_l\right]
        \\
        &+\sum_{i=1}^{m-1}\sum_{j=i+1}^m\sum_{l=i+1,l\neq j}^m(j-i+2)(l-i+2)\E\left[I_i^2I_jI_l\right]
        \\
        &+ \sum_{i=1}^{m-1}\sum_{j=i+1}^m\sum_{l=j+1}^m(j-i+2)(l-j+2)\E\left[I_iI_j^2I_l\right]
        \\
        &+ \sum_{i=1}^{m-1}\sum_{j=i+1}^m(j-i+2)^2\E\left[I_i^2I_j^2\right]
        \\
        &+ \sum_{i=1}^{m-1}\sum_{j=i+1}^m\sum_{k=1}^{i-1}(j-i+2)(i-k+2)\E\left[I_i^2I_jI_k\right]
        \\
        &+\sum_{i=1}^{m-1}\sum_{j=i+1}^m\sum_{k=1}^{j-1}(j-i+2)(j-k+2)\E\left[I_iI_j^2I_k\right].
    \end{split}
    \end{align*}
    The calculations have led to terms where the expectations and the sums can be treated separately, since $I_1,\dots, I_m^{(n)}$ are identically distributed. Calculating the sum terms, we get
    \begin{align*}
        &\sum_{i=1}^{m-1}\sum_{j=i+1}^m\sum_{l=i+1,l\neq j}^m(j-i+2)(l-i+2)= \frac{m(m-1)(m-2)(3m^2+34m+111)}{60},\\
        &\sum_{i=1}^{m-1}\sum_{j=i+1}^{m}\sum_{l=j+1}^{m}(j-i+2)(l-j+2)=\frac{m(m-1)(m-2)(m+6)(m+17)}{120},\\
        &\sum_{i=1}^{m-1}\sum_{j=i+1}^{m}(j-i+2)^2 = \frac{m(m-1)(m^2+9m+32)}{12}\\
    \end{align*}
    and
    \begin{align*}
    \begin{split}
        \sum_{i=1}^{m-1}\sum_{j=i+1}^m\sum_{k=1,k\neq i,j}^{m-1}&\sum_{l=k+1,l\neq i,j}^m(j-i+2)(l-k+2)\\
        &=\frac{m(m-1)(m-2)(m-3)(5m^2+69m+244)}{180}.
        \end{split}
    \end{align*}
For the mixed moments of the multinomially distributed random variable $I^{(n)}$, we have
\begin{align}
    \begin{split}
    &\E\left[I_1 I_2 I_3 I_4\right]=\frac{n(n-1)(n-2)(n-3)}{m^4},\\
    &\E\left[\left(I_1\right)^2 I_2 I_3\right]=\frac{n(n-1)(n-2)(n-3)}{m^4}+\frac{n(n-1)(n-2)}{m^3},\\
    &\E\left[\left(I_1\right)^2 \left(I_2\right)^2\right]=\frac{n(n-1)}{m^2}\left(\frac{(n-2)(n-3)}{m^2}+\frac{2(n-2)}{m}+1\right).
    \end{split}
\end{align}
Putting this together, we obtain
\begin{align*}
    \begin{split}
        \E&\left[\left(\sum_{i=1}^{m-1}\sum_{j=i+1}^{m}(j-i+2)I_iI_j\right)\left(\sum_{k=1}^{m-1}\sum_{l=k+1}^{m}(l-k+2)I_kI_l\right)\right]\\
        &= \frac{n(n-1)(n-2)(n-3)}{m^2}  \frac{(m-1)^2(m+7)^2}{36}
        \\
        &+ \frac{n(n-1)(n-2)}{m^2}  \frac{(m-1)(7m^3+87m^2+232m-328)}{60}
        \\
        &+ \frac{n(n-1)}{m}  \frac{(m-1)(m^2+9m+32)}{12}.
    \end{split}
\end{align*}
The variance can now be computed by
\begin{align*}
\begin{split}
    \Var&\left(\sum_{i<j} (j-i+2)I_iI_j\right)\\
    &=\E\left[\left(\sum_{i=1}^{m-1}\sum_{j=i+1}^{m}(j-i+2)I_iI_j\right)\left(\sum_{k=1}^{m-1}\sum_{l=k+1}^{m}(l-k+2)I_kI_l\right)\right]\\
    &\quad-\E\left[\sum_{i=1}^{m-1}\sum_{j=i+1}^{m}(j-i+2)I_iI_j\right]^2\\
    &=\frac{n^3(m-1)(m-2)(m+1)(m+2)}{180m^2}+\frac{n^2(m-1)(m^3+m^2+71m+251)}{90m^2}\\
    &\quad-\frac{n(m-1)(m^3+m^2+46m+166)}{60m^2}.
    \end{split}
\end{align*}
\end{proof}
\noindent
The second term is relatively easy compared to the others since there are no products of $I_i$ to be estimated.
\begin{lem}
We have
\begin{align}
\Var\left(\sum_{i,j=1}^m \left(\vert j-i \vert +1\right)I_j\right)=\frac{n(m-1)(m-2)(m+1)(m+2)}{180}.
\end{align}
\end{lem}
\begin{proof}
    The calculation of the expectation yields us
    \begin{align*}
    \begin{split}
        \E\left[\sum_{i,j=1}^m \left(\vert j-i \vert +1\right)I_j\right]&=\sum_{j=1}^m\left(j^2-j(m+1)+\frac{m(m+3)}{2}\right)\E\left[I_j\right]\\
        &= \frac{n(m^2+3m-1)}{3}.
        \end{split}
    \end{align*}
Furthermore, we get for the second moment
\begin{align*}
\begin{split}
      \E&\left[\left(\sum_{i,j=1}^m \left(\vert j-i \vert +1\right)I_j\right)^2\right]\\
      &=\sum_{i=1}^m\sum_{j=1}^m\left(i^2-i(m+1)+\frac{m(m+3)}{2}\right)\left(j^2-j(m+1)+\frac{m(m+3)}{2}\right)\E\left[I_iI_j\right]\\
      &= \sum_{i=1}^m\sum_{j=1,j \neq i}^m\left(i^2-i(m+1)+\frac{m(m+3)}{2}\right)\left(j^2-j(m+1)+\frac{m(m+3)}{2}\right)\E\left[I_iI_j\right]
        \\
        &\quad +\sum_{j=1}^m\left(j^2-j(m+1)+\frac{m(m+3)}{2}\right)^2\E\left[I_i^2\right]\\
        &=\frac{n(n-1)(m^2+3m-1)^2}{9} + \frac{n(7m^4+40m^3+45m^2-40m+8)}{60}
      \end{split}
\end{align*}
using $\E[I_1I_2]=n(n-1)/m^2$ and $\E[(I_1)^2]=n(n-1)/m^2+n/m$. Totally, we obtain
\begin{align*}
    Var\left(\sum_{i,j=1}^m \left(\vert j-i \vert +1\right)I_j\right)=\frac{n(m-1)(m-2)(m+1)(m+2)}{180}.
\end{align*}
\end{proof}
\begin{lem}
    We have
    \begin{align}
    \Var\left(\sum_{i=1}^m I_i\left(I_i-1\right)\right)=\frac{2n(n-1)(m-1)}{m^2}.
    \end{align}
\end{lem}
\begin{proof}
    The expectation gives us
    \begin{align*}
        \E\left[\sum_{i=1}^m I_i\left(I_i-1\right)\right] & 
        = m\left(\E\left[I_1^2\right]-\E\left[I_1\right]\right)\\
        &= m\bigg(\frac{n(n-1)}{m^2}+\frac{n}{m}-\frac{n}{m}\bigg) = \frac{n(n-1)}{m}.
    \end{align*}
    For the second moment we obtain
    \begin{align*}
        \E\left[\left(\sum_{i=1}^m I_i\left(I_i-1\right)\right)^2\right]&=m(m-1)\E\left[\left(I_1^2-I_2\right)\left(I_2^2-I_2\right)\right]\\
        &\quad+m\E\left[\left(I_1^2-I_1\right)^2\right]
    \end{align*}
    Using the mixed moments of a multinomial distribution, we get
    \begin{align*}
        \E\left[\left(I_1^2-I_1\right)\left(I_2^2-I_2\right)\right] &= \E\left[I_1^2I_2^2\right]-2\E\left[I_1^2I_2\right]+\E\left[I_1I_2\right]
        \\
        &= \left(\frac{n(n-1)(n-2)(n-3)}{m^4}+\frac{2n(n-1)(n-2)}{m^3}+\frac{n(n-1)}{m^2}\right)
        \\
        &\quad-2\left(\frac{n(n-1)(n-2)}{m^3}+\frac{n(n-1)}{m^2}\right)+\frac{n(n-1)}{m^2}
        \\
        &=\frac{n(n-1)(n-2)(n-3)}{m^4}
        \\
    \end{align*}
    and
    \begin{align*}
        \E\left[\left(I_1^2-I_1\right)\right] &= \E\left[I_1^4\right]-2\E\left[I_1^3\right]+\E\left[I_1^2\right]
        \\
        &= \left(\frac{n(n-1)(n-2)(n-3)}{m^4}+\frac{6n(n-1)(n-2)}{m^3}+\frac{7n(n-1)}{m^2}+\frac{n}{m}\right)
        \\
        &\quad- 2\bigg(\frac{n(n-1)(n-2)}{m^3}+\frac{3n(n-1)}{m^2}+\frac{n}{m}\bigg) + \bigg(\frac{n(n-1)}{m^2}+\frac{n}{m}\bigg)
        \\
        &= \frac{n(n-1)(n-2)(n-3)}{m^4}+\frac{4n(n-1)(n-2)}{m^3}+\frac{2n(n-1)}{m^2}.
    \end{align*}
    This gives us 
    \begin{align*}
        \E\left[\left(\sum_{i=1}^m I_i\left(I_i-1\right)\right)^2\right]=&\frac{n(n-1)(n-2)(n-3)+4n(n-1)(n-2)}{m^2}\\
        &+\frac{2n(n-1)}{m}.
    \end{align*}
    Putting the first and second moment together we obtain
    \begin{align*}
        \Var&\left(\sum_{i=1}^m I_i\left(I_i-1\right)\right)=\E\left[\left(\sum_{i=1}^m I_i\left(I_i-1\right)\right)^2\right]-\E\left[\sum_{i=1}^m I_i\left(I_i-1\right)\right]^2\\
        &=\frac{n(n-1)\left((n-2)(n-3)-n(n-1)\right)+4n(n-1)(n-2)}{m^2} +\frac{2n(n-1)}{m}
        \\
        &= -\frac{2n(n-1)}{m^2}+\frac{2n(n-1)}{m} = \frac{2n(n-1)(m-1)}{m^2}
    \end{align*}
\end{proof}
\noindent
In a second step, we have to calculate the covariances between the three terms to obtain the variance of the sum.
\begin{lem}
    We have
    \begin{align*}
        \Cov\bigg(\sum_{i<j} (j-i+2)I_iI_j,&\sum_{k,\ell=1}^m \left(\vert \ell-k \vert +1\right)I_\ell\bigg)\\
        &=\frac{n(n-1)(m-1)(m-2)(m+1)(m+2)}{180m}.
    \end{align*}
\end{lem}
\begin{proof}
    Since the expectations of the two terms are calculated beforehand, we only have to derive an expression for the expectation of the product. We have
    \begin{align*}
        \E&\left[\left(\sum_{i<j} (j-i+2)I_iI_j\right)\left(\sum_{k,\ell=1}^m \left(\vert \ell-k \vert +1\right)I_\ell \right)\right]\\
        &=\sum_{i=1}^{m-1}\sum_{j=i+1}^{m}(j-i+2)\sum_{k=1,k\neq i,j}^m\left(k^2-k(m+1)+\frac{m(m+3)}{2}\right)\E[I_iI_jI_k]
        \\
        &\quad+\sum_{i=1}^{m-1}\sum_{j=i+1}^{m}(j-i+2)\left(i^2-i(m+1)+\frac{m(m+3)}{2}\right)\E[I_i^2I_j]
        \\
        &\quad+\sum_{i=1}^{m-1}\sum_{j=i+1}^{m}(j-i+2)\left(j^2-j(m+1)+\frac{m(m+3)}{2}\right)\E[I_iI_j^2]
    \end{align*}
The mixed moments are once again used from \cite{ouimet2020}. Furthermore, we get
\begin{align*}
    \sum_{i=1}^{m-1}\sum_{j=i+1}^{m}(j-i+2)&\sum_{k=1,k\neq i,j}^m\left(k^2-k(m+1)+\frac{m(m+3)}{2}\right)\\
    &=\frac{m(m-1)(m-2)(10m^3+99m^2+197m-72)}{180}
\end{align*}
and
\begin{align*}
    \sum_{i=1}^{m-1}\sum_{j=i+1}^{m}(j-i+2)&\left(i^2-i(m+1)+\frac{m(m+3)}{2}\right)\\
    &=\frac{m(m-1)(7m^3+67m^2+132m-48)}{120}.
\end{align*}
This gives us
\begin{align*}
    \E\Bigg[\Bigg(\sum_{i<j} (j-i+2)&I_iI_j\Bigg)
    \left(\sum_{k,\ell=1}^m \left(\vert \ell-k \vert +1\right)I_\ell \right)\Bigg]\\
    &=\frac{(m-1)(m+7)(m^2+3m-1)n(n-1)(n-2)}{18m}\\
    &\quad+ \frac{(m-1)(7m^3+67m^2+132m-48)n(n-1)}{60m},
\end{align*}
from which we can directly obtain the covariance
\begin{align*}
    \Cov\bigg(\sum_{i<j} (j-i+2)&I_iI_j,\sum_{k,\ell=1}^m \left(\vert \ell-k \vert +1\right)I_\ell\bigg)\\
    &=\frac{n(n-1)(n-2)(m-1)(m+7)(m^2+3m-1)}{18m} 
        \\
        &\quad+ \frac{n(n-1)(m-1)(7m^3+67m^2+132m-48)}{60m}
        \\
        &\quad- \frac{n(n-1)(m-1)(m+7)}{6m}\frac{n(m^2+3m-1)}{3}
        \\
        &= \frac{n(n-1)(m-1)(m-2)(m+1)(m+2)}{180m}.
\end{align*}
\end{proof}
\begin{lem}
    We have
    \begin{align*}
        \Cov\left(\sum_{i<j} (j-i+2)I_iI_j,\sum_{k=1}^m I_k\left(I_k-1\right)\right)=-\frac{2n(n-1)(m-1)(m+7)}{6m^2}.
    \end{align*}
\end{lem}
\begin{proof}
    We get
    \begin{align*}
        \begin{split}
            \E\Bigg[\sum_{i<j} (j-i+2)I_iI_j,&\sum_{k=1}^m I_k\left(I_k-1\right)\Bigg]\\
            &=\sum_{i=1}^{m-1}\sum_{j=i+1}^{m}\sum_{k=1}^m(j-i+2)\E\left[I_i I_j(I_k^2-I_k)\right]
        \\
        &= \sum_{i=1}^{m-1}\sum_{j=i+1}^{m}\sum_{k=1,k \neq i,j}^m(j-i+2)\E\left[I_iI_{j}(I_k^2-I_k)\right] 
        \\
        &\quad+\sum_{i=1}^{m-1}\sum_{j=i+1}^{m}(j-i+2)\E\left[I_iI_{j}(I_i^2-I_i)\right]\\
        &\quad+\sum_{i=1}^{m-1}\sum_{j=1}^{m-i}(j-i+2)\E\left[I_iI_{j}(I_{j}^2-I_{j})\right] 
        \end{split}
    \end{align*}
    For the sum terms we obtain
    \begin{align*}
        \begin{split}
            &\sum_{i=1}^{m-1}\sum_{j=1}^{m-i}{j-i+2}=\frac{m(m-1)(m+7)}{6},\\
            &\sum_{j=i+1}^{m}\sum_{k=1,k \neq i,j}^m {j-i+2}=(m-2)\sum_{i=1}^{m-1}\sum_{j=1}^{m-i}{j-i+2}=\frac{m(m-1)(m-2)(m+7)}{6}.
        \end{split}
    \end{align*}
    The expectations can be calculated by
    \begin{align*}
        \E\left[I_1I_2(I_3^2-I_3)\right]&=\E\left[I_1 I_2 I_3^2 \right]-\E[I_1 I_2 I_3]\\
        &=\left(\frac{n(n-1)(n-2)(n-3)}{m^4}+\frac{n(n-1)(n-2)}{m^3}\right) - \frac{n(n-1)(n-2)}{m^3}
        \\
        &= \frac{n(n-1)(n-2)(n-3)}{m^4}
        \\
        \E[I_1I_2(I_1^2-I_1)] &= \E\left[I_1^3I_2\right]-\E\left[I_1^2 I_2\right]
        \\
        &= \left(\frac{n(n-1)(n-2)(n-3)}{m^4}+\frac{3n(n-1)(n-2)}{m^3}+\frac{n(n-1)}{m^2}\right) 
        \\
        &\quad-\left(\frac{n(n-1)(n-2)}{m^3}+\frac{n(n-1)}{m^2}\right)
        \\
        &= \frac{n(n-1)(n-2)(n-3)}{m^4} +\frac{2n(n-1)(n-2)}{m^3}
    \end{align*}
    Putting all this together, we get
    \begin{align*}
        \E\Bigg[\sum_{i<j} (j-i+2)I_iI_j,&\sum_{k=1}^m I_k\left(I_k-1\right)\Bigg]\\
        &=\frac{n(n-1)(n-2)(n-3)}{m^4}\frac{m^2(m-1)(m+7)}{6}\\
        &\quad+\frac{2n(n-1)(n-2)}{m^3}\frac{2m(m-1)(m+7)}{6} 
        \\
        &=\frac{n(n-1)(n-2)(n-3)(m-1)(m+7)}{6m^2}\\
        &\quad+\frac{4n(n-1)(n-2)(m-1)(m+7)}{6m^2}\\
        &= \frac{n(n-1)(n-2)(n+1)(m-1)(m+7)}{6m^2}
    \end{align*}
    and therefore we obtain
    \begin{align*}
        \begin{split}
            \Cov\Bigg(\sum_{i<j} (j-i+2)I_iI_j,&\sum_{k=1}^m I_k\left(I_k-1\right)\Bigg)\\
            &=\E\left[\sum_{i<j} (j-i+2)I_iI_j,\sum_{k=1}^m I_k\left(I_k-1\right)\right]\\
            &\quad-\E\left[\sum_{i<j} (j-i+2)I_iI_j\right]\E\left[\sum_{k=1}^m I_k\left(I_k-1\right)\right]\\
            &=-\frac{2n(n-1)(m-1)(m+7)}{6m^2}
        \end{split}
    \end{align*}
\end{proof}
\begin{lem}
    We have
    \begin{align*}
        \Cov\left(\sum_{i,j=1}^m \left(\vert i-j \vert +1\right)I_j,\sum_{k=1}^m I_k(I_k-1)\right)=0.
    \end{align*}
\end{lem}
\begin{proof}
    We get
    \begin{align*}
        \E\Bigg[&\left(\sum_{i,j=1}^m \left(\vert i-j \vert +1\right)I_j\right)\left(\sum_{k=1}^m I_k(I_k-1)\right)\Bigg]\\
        &=\sum_{i=1}^m\sum_{j=1,j\neq i}^m\left(j^2-j(m+1)+\frac{m(m+3)}{2}\right)\E[I_i(I_i-1)I_j] 
        \\
        &\quad+\sum_{i=1}^m\left(i^2-i(m+1)+\frac{m(m+3)}{2}\right)\E[I_i^2(I_i-1)]
        \\
        &= (m-1)\frac{m(m^2+3m-1)}{3}\E[I_i(I_i-1)I_j]\\
        &\quad+\frac{m(m^2+3m-1)}{3}\E[I_i^2(I_i-1)]
    \end{align*}
Using
\begin{align*}
    \E[I_1(I_1-1)I_2] &= \E\left[I_1^2I_2\right]-\E[I_1I_2]
        \\
        &= \left(\frac{n(n-1)(n-2)}{m^3}+\frac{n(n-1)}{m^2}\right) - \frac{n(n-1)}{m^2}\\
        &= \frac{n(n-1)(n-2)}{m^3}
\end{align*}
and
\begin{align*}
    \E\left[I_1^2(I_1-1)\right] &= \E\left[I_1^3\right]-\E\left[I_1^2\right]\\
        &= \bigg(\frac{n(n-1)(n-2)}{m^3}+\frac{3n(n-1)}{m^2}+\frac{n}{m}\bigg) -\bigg(\frac{n(n-1)}{m^2}+\frac{n}{m}\bigg)
        \\
        &= \frac{n(n-1)(n-2)}{m^3}+\frac{2n(n-1)}{m^2}
\end{align*}
we obtain
\begin{align*}
    \E\Bigg[&\left(\sum_{i,j=1}^m \left(\vert i-j \vert +1\right)I_j\right)\left(\sum_{k=1}^m I_k(I_k-1)\right)\Bigg]\\
    &=\frac{m(m-1)(m^2+3m-1)}{3}\frac{n(n-1)(n-2)}{m^3}\\
    &\quad+\frac{m(m^2+3m-1)}{3}\left(\frac{n(n-1)(n-2)}{m^3}+\frac{2n(n-1)}{m^2}\right)\\
    &= \frac{n^2(n-1)(m^2+3m-1)}{3m}.
\end{align*}
This gives us
\begin{align*}
    \Cov&\left(\sum_{i,j=1}^m \left(\vert i-j \vert +1\right)I_j,\sum_{k=1}^m I_k(I_k-1)\right)\\
    &=\E\left[\left(\sum_{i,j=1}^m \left(\vert i-j \vert +1\right)I_j\right)\left(\sum_{k=1}^m I_k(I_k-1)\right)\right]\\
    &\quad - \E\left[\sum_{i,j=1}^m \left(\vert i-j \vert +1\right)I_j\right]\E\left[\sum_{k=1}^m I_k(I_k-1)\right]\\
    &=\frac{n^2(n-1)(m^2+3m-1)}{3m} -\frac{n(n-1)}{m}\frac{n(m^2+3m-1)}{3} = 0
\end{align*}
\end{proof}
\noindent
Now we have all tools to calculate the variance of $W_n$.
\begin{proof}[Proof of Theorem \ref{clt:Wiener} Part 1]
The variance is derived by putting the calculations done in the lemmata before together. We have
\begin{align*}
    \Var\left(W_n\right)&= 
        \frac{n^3(m-1)(m-2)(m+1)(m+2)}{180m^2}
        +\frac{n^2(m-1)(m^3+m^2+71m+251)}{90m^2}\\
        &\quad-\frac{n(m-1)(m^3+m^2+46m+166)}{60m^2}+\frac{2n(n-1)(m-1)}{m^2}\\
        &\quad+\frac{n(m-1)(m-2)(m+1)(m+2)}{18}-\frac{4n(n-1)(m-1)(m+7)}{6m^2}\\
        &\quad+\frac{2n(n-1)(m-1)(m-2)(m+1)(m+2)}{180m}\\
        &=\frac{n^3(m-1)(m-2)(m+1)(m+2)}{180m^2}\\
        &\quad+\frac{n^2(m-1)(m+1)(m^3+m^2-4m+11)}{90m^2}\\
        &\quad+\frac{n(m-1)(m+1)(m^4-2m^3-7m^2+8m-18)}{180m^2}.
\end{align*}
\end{proof}
\begin{rem}
Using analogous methods, we can calculate the variance of $R_n^{[1]}$. However, since the variance is not required for the proof of the limit theorems for the $\alpha$-Randić index, we will not present all the computations.
\end{rem}
\end{document}